\documentclass[12pt]{article}
\usepackage{amssymb,amsfonts,amsthm,amsmath,calligra}
\usepackage{slashed}
\usepackage{yfonts}
\usepackage{mathrsfs,pifont}

\usepackage[all]{xy}

\usepackage{amssymb,amsfonts,amsthm,amsmath}
\usepackage[all]{xy}
\usepackage{slashed}
\usepackage{yfonts}
\usepackage{mathrsfs,pifont}

\usepackage{color}

\usepackage{tikz}
\usetikzlibrary{plotmarks}
\usepackage{eucal}
\usepackage{tkz-berge}

\usepackage{accents}
\newlength{\dhatheight}
\newcommand{\doublehat}[1]{%
    \settoheight{\dhatheight}{\ensuremath{\widehat{#1}}}%
    \addtolength{\dhatheight}{-0.35ex}%
    \widehat{\vphantom{\rule{1pt}{\dhatheight}}%
    \smash{\widehat{#1}}}}

\def\hZ {{\textbf{Z}}}

\def\matZ{{\mathbb{Z}}}

\def\M{ {\textbf{M}}  }

\def\gt{{U_{\hbar}(\doublehat{\frak{gl}}_1)}}

\def\tb {{\cal{V}}} %tautological Bundles V on a Nakajima variety
\def\tw {{\cal{W}}} % tautological bundles W on a Nakajima variety
\def\qmV {{\mathscr{V}}} %quasimap tautological bundles V on P^1
\def\qm {\textsf{QM}} % moduli space of quasimaps
\def \vss {\widehat{{\cO}}_{\textrm{vir}}} % virtual structure sheaf on quasimaps
\def \ev {{\textrm{ev}}} % evaluation map
\def \fh {\frak{h}} % Heisenberg agebra
\def \Rtot {{\mathscr{R}}} % total R-matrix
\def \Rwal {{\textsf{R}}} % wall R-matrix
\def \gM {{\textbf{M}}_{{\cal{O}}(1)}} % geometric difference operator
\def \gJ {{\textbf{J}}} % Fusion operator
\def \fock {{\cal{F}}} %Fock space
\def \lb {{\overrightarrow{\lambda}}} %Fixed point on the instanton moduli
\def\dd {{\overrightarrow{d}}} % degree of a quasimap
\def\cv {{\hat{V}}}  % capped vertex
\def\fp {{\textbf{p}}} % fixed point on the moduli space of quasimaps
\def\be{\begin{eqnarray}}
\def\ee{\end{eqnarray}}

\newcommand{\C}{\mathbb{C}}

\newcommand{\Z}{\mathbb{Z}}
\newcommand{\R}{\mathbb{R}}
\newcommand{\Q}{\mathbb{Q}}

\newcommand{\cO}{\mathscr{O}}

\newcommand{\cM}{\mathscr{M}}

\newcommand{\bC}{\mathsf{C}}
\newcommand{\bA}{\mathsf{A}}
\newcommand{\bT}{\mathsf{T}}
\newcommand{\bG}{\mathsf{G}}

\newcommand{\fC}{\mathfrak{C}}

\DeclareMathOperator{\Hom}{\mathscr{H}\text{\kern -3pt {\calligra\large om}}\,}

\DeclareMathOperator{\Stab}{Stab}

\newtheorem{Corollary}{Corollary}
\newtheorem{Theorem}{Theorem}
\newtheorem{Lemma}{Lemma}
\newtheorem{Proposition}{Proposition}
\newtheorem{Conjecture}{Conjecture}

\theoremstyle{definition}

\begin{document}

\title{Rationality of capped descendent vertex in $K$-theory}
\author{Andrey  Smirnov}
\date{}
\maketitle

\begin{abstract}
In this paper we analyze the fundamental solution of the \textit{quantum difference equation} (qde) for the
moduli space of instantons on two-dimensional projective space.
The qde is a $K$-theoretic generalization of the quantum differential equation
in quantum cohomology.
As in the quantum cohomology case, the fundamental solution of qde provides the capping operator in $K$-theory  (the rubber part of the capped vertex). We study the dependence of the capping operator on the equivariant parameters $a_i$ of the torus acting on the instanton moduli space by changing the framing. We prove that the capping operator factorizes at $a_i\to 0$. The rationality of the $K$-theoretic 1-leg capped descendent vertex follows from factorization of the capping operator as a simple corollary.
\end{abstract}

\setcounter{tocdepth}{2}
\tableofcontents

\section{Introduction}
\subsection{Summary}
The rationality of partition functions in the presence of descendents is an important and long standing conjecture
in enumerative geometry of 3-folds. The deepest insight into this conjecture was achieved in the theory of stable pairs where the conjecture was proved for nonsingular toric 3-folds and local curve 3-folds \cite{PPRationality,PPToric} (the theory of stable pairs is conjecturally equivalent to Gromov-Witten theory and cohomological Donaldson-Thomas theory of 3-folds \cite{MOOP}). The central result here is the rationality of the cohomological capped 1-leg descendent vertex, see Theorem  3 in \cite{PPRationality}, which was obtained by detailed analysis of the poles of the descendent vertex. The rationality of the partition functions for local curve 3-folds and nonsingular toric 3-folds can be derived from this result by applying standard arguments such as degeneration and geometric reduction of 3-leg descendent vertex to the case of 1-leg.

In this paper we prove the rationality of the capped $K$-theoretic 1-leg descendent vertex - Theorem \ref{ratth}. This function is defined in the theory of stable quasimaps to the Hilbert scheme of points on the complex plane. The enumerative geometry of stable quasimaps  developed in \cite{pcmilect} is a $K$-theoretic generalization of quantum cohomology. We should stress here that our proof of rationality is completely different (and, we believe, simpler) than one in \cite{PPRationality}. Our main idea is to analyze the capped $K$-theory vertex in broader context: instead of restricting ourself to quasimaps to the Hilbert scheme of points $\textrm{Hilb}^{n}(\C^2)$ we consider the quasimaps to the moduli of instantons $\cM(n,r)$ of arbitrary rank $r$ and topological charge $n$. The Hilbert schemes are covered by a special case $r=1$ when $\cM(n,1)=\textrm{Hilb}^{n}(\C^2)$. For a general $r$ the coefficients of power series for both the bare vertex and the capping operator are rational functions of additional equivariant parameters $a_1\cdots a_r$ corresponding to the action of framing torus $\bA=(\C^{*})^r$ on $\cM(n,r)$.
We analyze the asymptotic behaviour of the capping operator and vertex in the limit $a_i\to 0$ and find that these functions factorize as stated in  Theorem \ref{facth} and Theorem \ref{facver}. In Section \ref{pfrat} we show that the rationality of the capped descendent vertex follows in an elementary way from these two factorization theorems.

The main ingredient of our approach is the \textit{quantum difference equation} (qde), see equation (\ref{qde1}). As explained in \cite{OS}, in $K$-theory this equation plays the role similar to one of the quantum differential equation in quantum cohomology
of instanton moduli space $\cM(n,r)$.
The capping operator in $K$-theory is given by the fundamental solution of qde.
We note here, that in the special case $r=1$, the cohomological limit of the operator $\gM(z)$ in qde (\ref{qde1}) coincides with the operator of quantum multiplication by the first Chern class of tautological bundle in the quantum cohomology of $\textrm{Hilb}^{n}(\C^2)$ given by formula (6) in \cite{OPQuanCoh}.
Therefore, in this limit the qde turns to the quantum differential equation for $\textrm{Hilb}^{n}(\C^2)$ described in \cite{OP}. As a consequence, rationality of the descendent vertex in the quantum cohomology follows from our main Theorem \ref{ratth} through the cohomological limit.

The paper is organized as follows. First we recall some necessary facts about $\cM(n,r)$ and various tori acting on this moduli space. Following \cite{pcmilect} we define the bare and capped vertex functions in Section \ref{versec}. We then formulate the factorization Theorems \ref{facth} and \ref{facver} and prove the main Theorem \ref{ratth} as a corollary.

In Section \ref{torsec} we outline the theory of the quantum toroidal algebra $\gt$.
Using the results of \cite{OS} we then describe the action of this algebra on the equivariant $K$-theory of the instanton moduli spaces.
In Section \ref{qdedefsec} we give a universal formula for qde for instanton moduli space of arbitrary rank $r$ in terms of Heisenberg subalgebras
of $\gt$. We then prove the factorization Theorem \ref{facth} for the capping operator.
In Section  \ref{vers} we give an explicit formula for  the bare 1-leg $K$-theoretic vertex. We use it to prove a factorization theorem for the vertex.

\subsection{Instanton moduli \label{instantsec}}
Let $\cM(n,r)$ be the moduli space of framed rank $r$ torsion-free sheaves ${\cal{F}}$ on $\mathbb{P}^2$ with fixed second Chern class $c_{2}({\cal{F}})=n$. A framing of a sheaf ${\cal{F}}$ is a choice of an isomorphism:
\be
\label{fram}
\phi: \left.{\cal{F}}\right|_{L_{\infty}} \to {\mathscr{O}}^{\oplus r }_{L_{\infty}}
\ee
where $L_{\infty}$ is the line at infinity of $\C^{2} \subset \mathbb{P}^2$. This moduli space is usually referred to as rank $r$ instanton moduli space. Note that in the special case $r=1$ this moduli space is isomorphic to the Hilbert scheme of $n$-points on the complex plane
$$
\cM(n,1)=\textrm{Hilb}^{n}(\C^2).
$$

Let $\bA \simeq (\C^{\times})^r$ be the framing torus acting on $\cM(n,r)$ by scaling the $i$-th summand in isomorphism (\ref{fram}) with a character which we denote by $a_i$. This torus acts on the instanton moduli space preserving the symplectic form. Let us denote by $\bT=\bA\times (\C^{\times})^2$ where the second factor acts on $\C^2 \subset \mathbb{P}^2 $ by scaling the coordinates on the plane with characters which we denote by $t_1$ and $t_2$. This induces an action of $\bT$ on~$\cM(n,r)$. The action of this torus scales the symplectic form with a character $\hbar=t_1 t_2$.

Let $\bC \subset \bA$ be a one-dimensional subtorus acting on the $r=r_1+r_2$-dimensional framing space
with the character $r_1+a r_2$. In this situation we say that subtorus $\bC$ splits the framing $r$ to $r_1+a r_2$.
We note that the components of $\bC$-fixed point set are of the form:
\be
\cM(n,r)^{\bC}=\coprod\limits_{n_1+n_2=n}\, \cM(n_1,r_1)\times \cM(n_2,r_2).
\ee
If we set $\cM(r)=\coprod\limits_{n=0}^{\infty}\, \cM(n,r)$ then
\be
\label{cfps}
\cM(r)^{\bC}=\cM(r_1)\times \cM(r_2).
\ee
\subsection{Vertex functions \label{versec}}
Let $\qm^{d}(n,r)$ be the moduli space of stable quasimaps from $\mathbb{P}^1$ to $\cM(n,r)$~\footnote{In this paper we follow the terminology and notations of \cite{pcmilect}. A good introduction to the stable quasimaps is Section 4.3 of \cite{pcmilect}.}. Let us consider an action of a one-dimensional torus on $\mathbb{P}$ which comes from scaling the standard coordinate on $\mathbb{P}$. The fixed point set of this action consist of two points $\{p_1,p_2\}=\{0,\infty \} \subset \mathbb{P}$.  We denote the character of $T_{p_1} \mathbb{P}$ by $q$ and the torus by $\C^{*}_{q}$. This action induces an action of $\C^{*}_{q}$ on $\qm^{d}(n,r)$. We denote the total torus acting on $\qm^{d}(n,r)$ by $\bG=\bT\times \C^{*}_{q}$.

For a point $p\in \mathbb{P}$ let $\qm^{d}_{p}(n,r)
\subset \qm^{d}(n,r)$ be the open subset of quasimaps non-singular at $p$. This subset
comes together with the evaluation map:
$$
\ev_p: \qm^{d}_{p}(n,r) \longrightarrow \cM(n,r)
$$
sending a quasimap to its value at $p$.

The moduli space of relative quasimaps $ \widehat{\qm^{d}_{p}}(n,r)$ is a resolution of the map $\ev$ meaning that we have a commutative diagram:
\[
\xymatrix{
 &  \widehat\qm^{d}_{p}(n,r) \ar[dr]^{\widehat{\ev}_p} &  \\
 {\qm^{d}_{p}}(n,r) \ar@{^{(}->}[ur] \ar[rr]^{\ev_p} &  & \cM(n,r)
}
\]
with a \textit{proper} evaluation map $\widehat{\ev}_p$. The explicit construction of the  moduli space of relative quasimaps
is given in Section 6.4 of \cite{pcmilect}.

Let $L=\C^n$  and  $\tau \in K_{GL(L)}=\Lambda[x_1^{\pm 1},\cdots,x_n^{\pm 1}]$ be a symmetric Laurent polynomial. Every such polynomial corresponds to a virtual representation $\tau(L)$ of $GL(L)$ which is a tensorial polynomial in $L$. For example,
elementary symmetric function $\tau=e_{k}(x_i)$ corresponds to $\tau(L)=\Lambda^{\!k} L$. The \textit{bare vertex} with a descendent $\tau$ is defined by the following formal power series:
\be\label{bare}
V^{(\tau)}(z)=\sum\limits_{d=0}^{\infty}\,\ev_{p_2,*}\Big( \qm^{d}_{p_2}(n,r), \vss\otimes \tau(\left.\qmV\right|_{p_1})  \Big) z^d \in K_{\bG}\Big(\cM(n,r)\Big)_{loc}\otimes \Q[[z]]
\ee
Here $\qmV$ is the rank $n$ degree $d$ bundle on $\mathbb{P}^1$ defining the quasimap
and $\vss$ is the virtual structure of $\qm^{d}_{p_2}(n,r)$.  The pushforward $\ev_{p_2,*}$ is not proper; however, the fixed locus of $\bG$ is (in fact, the fixed locus is a set of finitely many isolated points in this case). Therefore, the pushforward is well defined in localized $K$-theory.

The \textit{capped vertex} with a descendent $\tau$ is a similar object defined for quasimap moduli space relative to $p_2$:
\be
\label{capped}
\hat{V}^{(\tau)}(z)=\sum\limits_{d=0}^{\infty}\,\widehat{\ev}_{p_2,*}\Big( \widehat{\qm}^{d}_{p_2}(n,r), \vss\otimes \tau(\left.\qmV\right|_{p_1})  \Big) z^d \in K_{\bG}\Big(\cM(n,r)\Big)\otimes \Q[[z]]
\ee
Now, $\widehat{\ev}_{p_2}$ is proper and the result lives in  non-localized $K$-theory.

By definition, the degree zero quasimaps  $\widehat{\qm}^{0}_{p_2}(n,r)=\cM(n,r)$ correspond to the constant maps from $\mathbb{P}$ to the instanton moduli space. In this case we have $\left.\qmV\right|_{p_1}=\tb$ where $\tb$ is a tautological bundle on $\cM(n,r)$.
By definition of the capped vertex :
$$
\hat{V}^{(\tau)}(z)=\tau(\tb) {\cal{K}}^{1/2}+ O(z)
$$
where ${\cal{K}}$ is the canonical bundle on $\cM(n,r)$.

 As an element of non-localized $K$-theory the capped vertex is a simpler object. For example at large $r$ all quantum corrections vanish:
\begin{Theorem} (Theorem 7.5.23 in \cite{pcmilect})
\label{largeth}
For every $\tau$ there exist $r\gg0$ such that $\hat{V}^{(\tau)}(z)=\tau(\tb) {\cal{K}}^{1/2}$.
\end{Theorem}
In fact, numerical computations shows that one can give the precise bound for $r$ in this theorem.
Assume that the descendent is given by a Schur polynomial $\tau=s_\lambda(x_1,\cdots,x_n)$ for a partition $\lambda=(\lambda_1\geq\lambda_2\geq\cdots \geq\lambda_k)$.
\begin{Conjecture}
For every $\tau=s_{\lambda}(x_1,\cdots,x_n)$ the capped vertex is classical $\hat{V}^{(\tau)}(z)=\tau(\tb) {\cal{K}}^{1/2}$ if and only if $r>\lambda_1$.
\end{Conjecture}
In particular, this conjecture implies that for the Hilbert schemes of points on a plane $\textrm{Hilb}^{n}(\C^2)$ corresponding to $r=1$ the capped vertex is classical only in the absence of descendents, i.e, for  $\tau=1$.
In general, all terms in the power series (\ref{capped}) are non-vanishing. However, our main theorem says that this power series is in fact a rational function:
\begin{Theorem}
\label{ratth}
The power series $\hat{V}^{(\tau)}(z)$ is the Taylor expansion in $z$ of a rational function in $\Q(u_1,\cdots,u_r,t_1,t_2,q,z)$.
\end{Theorem}

\subsection{Capping operator}
In computing the capped vertex one can separate the contributions of the $\qm^{d}_{p_2}(n,r)$ (which corresponds to the bare vertex) and the ``rubber'' part of the relative moduli space, see Section 7  of \cite{pcmilect}:
\be
\label{cver}
\cv^{(\tau)}(z)=\Psi(z) V^{(\tau)}(z).
\ee
Here
$$
\Psi(z)=K_{\bG}\Big(\cM(n,r)\Big)_{loc}^{\otimes 2}\otimes \Q[[z]]
$$
is the so called capping operator corresponding to the contribution of the rubber part. The matrix $\Psi(z)$ can be computed explicitly as a matrix of the \textit{fundamental solution of the quantum difference equation}:
\be
\label{qde1}
\Psi(z) {{\cal{O}}(1)}=\gM(z)\Psi(z)
\ee
where ${{\cal{O}}(1)}$ is an operator of multiplication by the corresponding line bundle in
$K_{\bT}(\cM(n,r))$. The operator $\gM(z)$ acts in $K_{\bT}(\cM(n,r))$ and has rational in $z$ matrix coefficients. It is constructed explicitly in Section \ref{qdeopsec}.
\subsection{Factorization theorems}
Assume that the torus $\bC$ splits the framing $r$ to $r_1+a r_2$,  so that the set of $\bC$-fixed points is given by (\ref{cfps}).
By definition, the coefficients of power series
$\Psi(z)$ and $V^{(\tau)}(z)$ are given by classes of localized $K$-theory  and thus are rational functions of $a$. We are interested in $a\to 0$ limits of these power series.
In Section \ref{focks} we describe an action of the quantum Heisenberg algebra $\fh$ on $K_{\bT}\Big(\cM(r)\Big)$.
Let $\alpha_k$, $k \in \Z \setminus\{0\}$ and $K$ be the standard generators of $\fh$. In Section \ref{psec} we prove the following result:
\begin{Theorem}
\label{facth}
If $\Psi^{(r)}(z)$ is the solution of (\ref{qde1}) for $\cM(r)$ then:
\be
\label{mainf}
\lim\limits_{a \to 0}\Psi^{(r)}(z)= Y^{(r_1),(r_2)}(z) \, \Psi^{(r_1)}(z \hbar^{\frac{r_2}{2}} ) \otimes \Psi^{(r_2)}(z \hbar^{-\frac{r_1}{2}})
\ee
where $Y^{(r_1),(r_2)}(z)$ is the evaluation of the following universal element
\be\label{uY}
Y(z)=\exp\Big( -\sum\limits_{k=1}^{\infty}\, \dfrac{(\hbar-\hbar^{-1}) K^k\otimes  K^{-k}}{1-z^{-k} K^k\otimes  K^{-k}} \,  \alpha_{-k} \otimes \alpha_{k}  \Big) \in \fh^{\otimes 2}\!(z)
\ee
in the $\fh^{\otimes 2}$ - representation $K_{\bT}\Big(\cM(r_1)\Big)\otimes K_{\bT}\Big(\cM(r_2)\Big)$.
\end{Theorem}
\noindent
%Here, the superscript in $\Psi^{(r)}(z)$ indicates the capping operator for $\cM(r,n)$.
%Thus, in this limit, up to a simple universal factor $Y(z)$ the capping operator
%factorizes to a product of capping operators for instantons moduli spaces with smaller ranks in accordance with the framing splitting
%$r=r_1+a r_2$.
Note that $ \alpha_{-k} \otimes \alpha_{k}$ act on the corresponding $K$-theory as locally nilpotent operators.
Thus, the coefficients of $Y^{(r_1),(r_2)}(z)$ are rational functions of~$z$ for all $r_1$ and $r_2$.
In Section \ref{vers}  we also prove a similar result for the bare vertex:
\begin{Theorem}
\label{facver}
For a descendent $\tau\in\Lambda[x_1,\cdots, x_n]$ we have\footnote{In full generality
the descendent $\tau$ can be a symmetric Laurent polynomial $\tau\in K_{GL(n)}(\cdot)= \Lambda[x_1^{\pm},\cdots, x_n^{\pm}]$, but it will be clear in the proof of this proposition that cases with inverse powers of $x_i$ are treated in the same way when $a\to \infty$. The result remains the same if we change the roles of $r_1$ and $r_2$.}:
\be
\lim\limits_{a\to 0} V^{(r),(\tau)}(z)=V^{(r_1),(\tau)}(z \hbar^{\frac{r_2}{2}}  ) \otimes V^{(r_2),(1)}(z \hbar^{-\frac{r_1}{2}} q^{-r_1}    )
\ee
\end{Theorem}
Here, as well as in Theorem \ref{facth}, the additional superscript $(r)$ corresponds to the ranks of the instanton moduli.

\subsection{Proof of the main Theorem \ref{ratth} \label{pfrat}}
The proof of Theorem \ref{ratth} is now elementary.  Indeed, for arbitrary $r_1$ and $\tau$ and in Theorem \ref{facver} let $\cv^{(r_1),(\tau)}(z)$ be the corresponding capped vertex. We need to check that $\cv^{(r_1),(\tau)}(z)$ is a rational function of $z$.  First, by Theorem \ref{largeth} we can find $r$ large enough for the corresponding capped descendent vertex to be classical $\cv^{(r),(\tau)}(z)=\tau(\tb) {\cal{K}}^{1/2}$, i.e. independent of~$z$.

We have:
$$
\tau(\tb) {\cal{K}}^{1/2}=\Psi^{(r)}(z) V^{(r),(\tau)}(z).
$$
The bundle ${\cal{K}}$  does not depend on $a$ and by our choice of $\tau$ we have $\lim\limits_{a\to 1} \tau(\tb)=\tau(\tb)\otimes 1$.
Thus, by the factorization theorems  we obtain:
$$
\begin{array}{l}
(\tau(\tb)\otimes 1) {\cal{K}}^{1/2}=\\
\\
Y^{(r_1),(r_2)}(z) \Psi^{(r_1)}(z \hbar^{\frac{r_2}{2}}  ) V^{(r_1),(\tau)}( z  \hbar^{\frac{r_2}{2}} ) \otimes \Psi^{(r_2)}(z \hbar^{-\frac{r_1}{2}} ) V^{(r_2),(1)}( z  \hbar^{-\frac{r_1}{2}} q^{-r_1} )
\end{array}
$$
Now, the first factor on the right side gives the capped vertex
with shifted parameter $\Psi^{(r_1)}(z \hbar^{\frac{r_2}{2}}) V^{(r_1),(\tau)}( z  \hbar^{\frac{r_2}{2}}  )=\cv^{(r_1),(\tau)}(z \hbar^{\frac{r_2}{2}})$.
The operator $Y^{(r_1),(r_2)}(z)$ from Theorem \ref{facth} is explicitly  invertible, therefore:
$$
\cv^{(r_1),(\tau)}(z \hbar^{\frac{r_2}{2}}) \otimes \Psi^{(r_2)}(z \hbar^{-\frac{r_1}{2}}) V^{(r_2),(1)}( z  \hbar^{-\frac{r_1}{2}} q^{-r_1} )=Y^{(r_1),(r_2)}(z)^{-1} (\tau(\tb)\otimes 1) {\cal{K}}^{1/2}
$$
The matrix $Y^{(r_1),(r_2)}(z)$ has rational coefficients, so  $Y^{(r_1),(r_2)}(z)^{-1}$ has. Thus, in the right side we have a vector whose components are rational functions of $z$. The first component of this vector is $\cv^{(r_1),(\tau)}(z \hbar^{\frac{r_2}{2}})$ and thus is also a rational function. Of course, the property to be rational does not depend on the shift, so  $\cv^{(r_1),(\tau)}(z)$ is a rational function of $z$. $\Box$
%\subsection{Final comment}
%Let us look at the last equation of the previous section at the so called quantum $K$-theory limit $q=1$. We note that
%in this case $ \Psi^{(r_2)}(z \hbar^{-\frac{r_1}{2}} ) V^{(r_2),(1)}( z  \hbar^{-\frac{r_1}{2}}) = {\cal{K}}^{1/2}$ and this equation takes the form:
%$$
%\cv^{(r_1),(\tau)}(z \hbar^{\frac{r_2}{2}}) \otimes 1=Y^{(r_1),(r_2)}(z)^{-1} (\tau(\tb){\cal{K}}^{1/2} \otimes 1)
%$$
%We can of course assume that $r\gg 0$ so that it holds for all $\tau$. This brings us to the following conclusion:
%there exists a universal operator acting on $K_{\bT}(\cM(\infty))$ with rational coefficients which \textit{maps the classical tautological classes to the corresponding capped descendent vertexes}. As such, this operator provides the most effective tool for computing the quantum descendents. The detailed analysis of this idea is, however, beyond the scope of this paper.

\subsection{Acknowledgements}
I would like to thank A.~Okounkov for guidance and his interest to this work.   
The author was also supported in part by RFBR grants 15-31-20484 mol-a-ved and RFBR 15-02-04175.

\section{Quantum toroidal algebra $\frak{gl}_1$ \label{torsec}}
\subsection{Generators and relations}
Let us set $\hZ =\matZ^2$, $\hZ^{*}=\hZ\ \{(0,0)\}$ and:
$$
\hZ^{+}=\{ (i,j)\in \hZ; i>0 \ \ \textrm{or} \ \ i=0,\ \ j>0 \}, \ \ \hZ^{-}=-\hZ^{+}
$$
Set
$$n_{k}=\dfrac{(t_1^{\frac{k}{2}}-t_1^{-\frac{k}{2}})(t_2^{\frac{k}{2}}-t_2^{-\frac{k}{2}})
(\hbar^{-\frac{k}{2}}-\hbar^{\frac{k}{2}})}{k}$$
 and for vector $\textbf{a}=(a_1,a_2) \in \hZ$ denote by $\textrm{deg}(\textbf{a})$ the greatest common divisor of $a_1$ and $a_2$. We set $\epsilon_\textbf{a}=\pm 1$ for $\textbf{a}\in \hZ^{\pm}$. For a pair non-collinear vectors we set $\epsilon_{\textbf{a},\textbf{b}}=\textrm{sign}(\det(\textbf{a},\textbf{b}))$.

\noindent

The ``toroidal'' algebra $\gt$ is an associative algebra with $1$ generated by elements $e_{\textbf{a}}$ and $K_\textbf{a}$ with $\textbf{a} \in \hZ$, subject to the following relations \cite{SchifVas}:
\begin{itemize}
\item  elements $K_\textbf{a}$ are central and
$$K_0=1, \ \ \ K_\textbf{a} K_\textbf{b}=K_{\textbf{a}+\textbf{b}}$$
\item if  $\textbf{a}$, $\textbf{b}$ are two collinear vectors then:
\be
\label{colin}
[e_\textbf{a},e_\textbf{b}]=\delta_{\textbf{a}+\textbf{b}} \dfrac{K^{-1}_\textbf{a}-K_\textbf{a}}{n_{\textrm{deg}(\textbf{a})}}
\ee
\item if $\textbf{a}$ and $\textbf{b}$ are such that $\textrm{deg}(\textbf{a})=1$ and the triangle $\{(0,0), \textbf{a}, \textbf{b}\}$
has no interior lattice points then
$$
[e_\textbf{a},e_\textbf{b}]=\epsilon_{\textbf{a},\textbf{b}} K_{\alpha(\textbf{a},\textbf{b})} \, \dfrac{\Psi_{\textbf{a}+\textbf{b}}}{n_1}
$$
where
$$
\alpha(\textbf{a},\textbf{b})=\left\{\begin{array}{ll}
\epsilon_{\textbf{a}}(\epsilon_{\textbf{a}} \textbf{a} +\epsilon_\textbf{b} \textbf{b} -\epsilon_{\textbf{a}+\textbf{b}}(\textbf{a}+\textbf{b})  )/2 & \textrm{if} \ \ \epsilon_{\textbf{a},\textbf{b}}=1\\
\epsilon_{\textbf{b}}(\epsilon_{\textbf{b}} \textbf{b} +\epsilon_\textbf{b} \textbf{b} -\epsilon_{\textbf{a}+\textbf{b}}(\textbf{a}+\textbf{b})  )/2 & \textrm{if} \ \ \epsilon_{\textbf{a},\textbf{b}}=-1
\end{array}\right.
$$
and elements $\Psi_\textbf{a}$ are defined by:
$$
\sum\limits_{k=0}^{\infty} \, \Psi_{k \textbf{a}} z^k = \exp\Big( \sum\limits_{m=1}^{\infty}\, n_m \, e_{m\, \textbf{a}} z^m \Big)
$$
for $\textbf{a}\in \hZ$ such that $\textrm{deg}(\textbf{a})=1$.
\end{itemize}
\hspace{2mm}
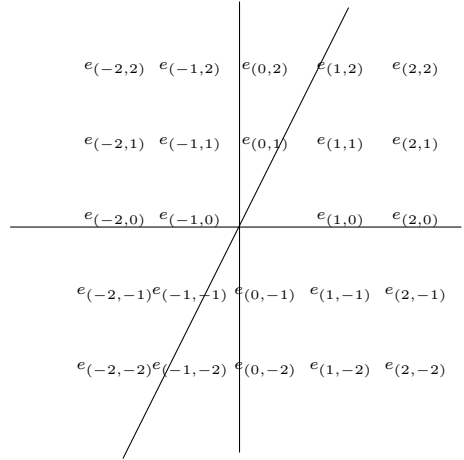
\begin{figure}[ht]
\begin{center}
\begin{tikzpicture}
\draw (-3,0)--(3,0);
\draw (-1.5,-3.08)--(1.5,2.92);
\draw (0.05,-3)--(0.05,3);
\node[mark size=1pt,color=black] at (2,-3) {\pgfuseplotmark{*}};
\node[mark size=1pt,color=black] at (2,-2) {\pgfuseplotmark{*}};
\node[mark size=1pt,color=black] at (2.4,-1.9) {\tiny{$e_{(2,-2)}$}};
\node[mark size=1pt,color=black] at (1,-1) {\pgfuseplotmark{*}};
\node[mark size=1pt,color=black] at (2.4,-0.9) {\tiny{$e_{(2,-1)}$}};
\node[mark size=1pt,color=black] at (2.4,0.1) {\tiny{$e_{(2,0)}$}};
\node[mark size=1pt,color=black] at (2,0) {\pgfuseplotmark{*}};
\node[mark size=1pt,color=black] at (2.4,1.1) {\tiny{$e_{(2,1)}$}};
\node[mark size=1pt,color=black] at (2,1) {\pgfuseplotmark{*}};
\node[mark size=1pt,color=black] at (2.4,2.1) {\tiny{$e_{(2,2)}$}};
\node[mark size=1pt,color=black] at (2,2) {\pgfuseplotmark{*}};
\node[mark size=1pt,color=black] at (2,3) {\pgfuseplotmark{*}};
\node[mark size=1pt,color=black] at (0,-3) {\pgfuseplotmark{*}};
\node[mark size=1pt,color=black] at (0,-2) {\pgfuseplotmark{*}};
\node[mark size=1pt,color=black] at (0.4,-1.9) {\tiny{$e_{(0,-2)}$}};
\node[mark size=1pt,color=black] at (0,-1) {\pgfuseplotmark{*}};
\node[mark size=1pt,color=black] at (0.4,-0.9) {\tiny{$e_{(0,-1)}$}};
\node[mark size=1pt,color=black] at (0,0) {\pgfuseplotmark{*}};
\node[mark size=1pt,color=black] at (0,1) {\pgfuseplotmark{*}};
\node[mark size=1pt,color=black] at (0.4,1.1) {\tiny{$e_{(0,1)}$}};
\node[mark size=1pt,color=black] at (0,2) {\pgfuseplotmark{*}};
\node[mark size=1pt,color=black] at (0.4,2.1) {\tiny{$e_{(0,2)}$}};
\node[mark size=1pt,color=black] at (0,3) {\pgfuseplotmark{*}};
\node[mark size=1pt,color=black] at (1,-3) {\pgfuseplotmark{*}};
\node[mark size=1pt,color=black] at (1,-2) {\pgfuseplotmark{*}};
\node[mark size=1pt,color=black] at (1.4,-1.9) {\tiny{$e_{(1,-2)}$}};
\node[mark size=1pt,color=black] at (1,-1) {\pgfuseplotmark{*}};
\node[mark size=1pt,color=black] at (1.4,-0.9) {\tiny{$e_{(1,-1)}$}};
\node[mark size=1pt,color=black] at (1,0) {\pgfuseplotmark{*}};
\node[mark size=1pt,color=black] at (1.4,0.1) {\tiny{$e_{(1,0)}$}};
\node[mark size=1pt,color=black] at (1,1) {\pgfuseplotmark{*}};
\node[mark size=1pt,color=black] at (1.4,1.1) {\tiny{$e_{(1,1)}$}};
\node[mark size=1pt,color=black] at (1,2) {\pgfuseplotmark{*}};
\node[mark size=1pt,color=black] at (1.4,2.1) {\tiny{$e_{(1,2)}$}};
\node[mark size=1pt,color=black] at (1,3) {\pgfuseplotmark{*}};
\node[mark size=1pt,color=black] at (-1,-3) {\pgfuseplotmark{*}};
\node[mark size=1pt,color=black] at (-1,-2) {\pgfuseplotmark{*}};
\node[mark size=1pt,color=black] at (-0.6,-1.9) {\tiny{$e_{(-1,-2)}$}};
\node[mark size=1pt,color=black] at (-1,-1) {\pgfuseplotmark{*}};
\node[mark size=1pt,color=black] at (-0.6,-0.9) {\tiny{$e_{(-1,-1)}$}};
\node[mark size=1pt,color=black] at (-1,0) {\pgfuseplotmark{*}};
\node[mark size=1pt,color=black] at (-0.6,0.1) {\tiny{$e_{(-1,0)}$}};
\node[mark size=1pt,color=black] at (-1,1) {\pgfuseplotmark{*}};
\node[mark size=1pt,color=black] at (-0.6,1.1) {\tiny{$e_{(-1,1)}$}};
\node[mark size=1pt,color=black] at (-1,2) {\pgfuseplotmark{*}};
\node[mark size=1pt,color=black] at (-0.6,2.1) {\tiny{$e_{(-1,2)}$}};
\node[mark size=1pt,color=black] at (-1,3) {\pgfuseplotmark{*}};

\node[mark size=1pt,color=black] at (-2,-3) {\pgfuseplotmark{*}};
\node[mark size=1pt,color=black] at (-2,-2) {\pgfuseplotmark{*}};
\node[mark size=1pt,color=black] at (-1.6,-1.9) {\tiny{$e_{(-2,-2)}$}};
\node[mark size=1pt,color=black] at (-2,-1) {\pgfuseplotmark{*}};
\node[mark size=1pt,color=black] at (-1.6,-0.9) {\tiny{$e_{(-2,-1)}$}};
\node[mark size=1pt,color=black] at (-2,0) {\pgfuseplotmark{*}};
\node[mark size=1pt,color=black] at (-1.6,0.1) {\tiny{$e_{(-2,0)}$}};
\node[mark size=1pt,color=black] at (-2,1) {\pgfuseplotmark{*}};
\node[mark size=1pt,color=black] at (-1.6,1.1) {\tiny{$e_{(-2,1)}$}};
\node[mark size=1pt,color=black] at (-2,2) {\pgfuseplotmark{*}};
\node[mark size=1pt,color=black] at (-1.6,2.1) {\tiny{$e_{(-2,2)}$}};
\node[mark size=1pt,color=black] at (-2,3) {\pgfuseplotmark{*}};

\node[mark size=1pt,color=black] at (-3,-3) {\pgfuseplotmark{*}};
\node[mark size=1pt,color=black] at (-3,-2) {\pgfuseplotmark{*}};
\node[mark size=1pt,color=black] at (-3,-1) {\pgfuseplotmark{*}};
\node[mark size=1pt,color=black] at (-3,0) {\pgfuseplotmark{*}};
\node[mark size=1pt,color=black] at (-3,1) {\pgfuseplotmark{*}};
\node[mark size=1pt,color=black] at (-3,2) {\pgfuseplotmark{*}};
\node[mark size=1pt,color=black] at (-3,3) {\pgfuseplotmark{*}};

\node[mark size=1pt,color=black] at (3,-3) {\pgfuseplotmark{*}};
\node[mark size=1pt,color=black] at (3,-2) {\pgfuseplotmark{*}};
\node[mark size=1pt,color=black] at (3,-1) {\pgfuseplotmark{*}};
\node[mark size=1pt,color=black] at (3,0) {\pgfuseplotmark{*}};
\node[mark size=1pt,color=black] at (3,1) {\pgfuseplotmark{*}};
\node[mark size=1pt,color=black] at (3,2) {\pgfuseplotmark{*}};
\node[mark size=1pt,color=black] at (3,3) {\pgfuseplotmark{*}};

\end{tikzpicture}
\caption{The line with slope $w=2$ corresponds to Heisenberg subalgebra
  generated by $\alpha_{k}^{2}=e_{k,2 k}$ for $k\in \Z\setminus \{0\}$.  \label{pic2}}
  \end{center}
\end{figure}

\subsection{Slope Heisenberg subalgebras}
For $w\in \Q \cup \{\infty\}$ we denote by $d(w)$ and $n(w)$ the denominator and numerator of rational number.
We set $d(\infty)=0$, $n(\infty)=1$ and $n(0)=0$, $d(0)=1$. Let us set:
$$
\alpha^{w}_{k}= e_{(d(w)k, n(w)k)}, \ \ k \in \Z \setminus \{0\}
$$
From (\ref{colin}) we see that for fixed $w\in \Q \cup \{\infty\}$ these elements generate a Heisenberg with the following relations:
$$
[\alpha^{w}_{-k},\alpha^{w}_{k}]=\dfrac{ K_{(1,0)}^{k d(w) } -K_{(1,0)}^{-k d(w) } }{n_k}
$$
We will informally refer to this algebra as ``Heisenberg subalgebra with a slope $w$'' and denote it $\frak{h}_w\subset \gt$.
It is convenient to visualize the algebra $\gt$ as in the Figure \ref{pic2}. Heisenberg subalgebras of $\frak{h}_w$
correspond to lines with slope $w$ in this picture.

The Heisenberg subalgebra with slope $w=0$ will play distinguished role in this paper. In this special case we will often omit the slope superscript: $\frak{h}=\frak{h}_0$, $\alpha_{k}=\alpha^{0}_{k}$  and $K=K_{(1,0)}$.

\subsection{Hopf structures \label{hopfsec}}
The algebra $\gt$ carried different inequivalent Hopf structures. We will use
the Hopf structure with zero slope defined by a coproduct $\Delta :\, \gt \to \gt\otimes \gt$, which has the following explicit form on Heisenberg subalgebra with slope zero:
\be
\begin{array}{l}
\Delta( \alpha_{-k} ) = \alpha_{-k}\otimes 1 + K^{-k} \otimes \alpha_{-k}\\
\\
\Delta( \alpha_{k} ) = \alpha_{k}\otimes K^k + 1 \otimes \alpha_{k}\\
\\
\Delta(K)=K\otimes K
\end{array}
\ee
where $k>0$. The algebra $\gt$ is a triangular Hopf algebra, which means there exist an element
$\Rtot\in \gt\otimes \gt$ (the universal $R$-matrix) which enjoys the following properties. In $\gt^{\otimes 3}$
it satisfies the quantum Yang-Baxter equation:
$$
\Rtot_{12}\Rtot_{13}\Rtot_{23}=\Rtot_{23}\Rtot_{13}\Rtot_{12}
$$
where indexes show in which component of $\gt^{\otimes 3}$ the corresponding $R$-matrix acts. In addition we have:
$$
1 \otimes \Delta (\Rtot) = \Rtot_{13}  \Rtot_{12}, \ \  \Delta \otimes 1 (\Rtot) = \Rtot_{13}  \Rtot_{23}
$$
and:
$$
\Rtot \Delta(g)=\Delta^{op}(g) \Rtot, \ \ \ \forall g \in \gt.
$$
where $\Delta^{op}$ is the opposite coproduct. Explicitly, the universal $R$-matrix is given by the following formula (Khoroshkin-Tolstoy factorization formula)
\be
\label{khtot}
\Rtot=\prod\limits_{w \in \Q  \cup\{\infty\}}^{\leftarrow} \, \exp\Big(\sum\limits_{k=0}^{\infty}\, n_k\, \alpha^{w}_{-k} \otimes \alpha^{w}_{k} \Big)
\ee
the order of factors in this product is given explicitly as \footnote{These two infinite products give the same element in the completion of $\gt^{\!\otimes 2}$. There is, however, an important difference - evaluated in the tensor product of two Fock representations
$\fock(a_1)\otimes \fock(a_2)$ the fist product converges in the topology of power series in $a_1/a_2$ while the second as power series in $a_2/a_1$. In this paper we are dealing with limits  $a_2/a_1\to 0$ of various operators and thus, use the second infinite product.}
:
$$
\Rtot=\prod\limits_{{w\in \Q}\atop {w<0}}^{\rightarrow} R^{-}_{w} \,R_{\infty}\, \prod\limits_{{w\in \Q}\atop {w  \geq 0}}^{\leftarrow} R^{+}_{w}=\prod\limits_{{w\in \Q}\atop {w \geq 0}}^{\leftarrow} \Big(R^{-}_{w}\Big)^{-1} \,\Big(R_{\infty}\Big)^{-1}\, \prod\limits_{{w\in \Q}\atop {w  < 0}}^{\rightarrow} \Big(R^{+}_{w}\Big)^{-1}
$$
where the order is the standard order on the set of rational numbers, and:
\be
\label{wallR}
R^{\pm}_{w}=\prod\limits_{k=0}^{\infty} \,\exp( n_k\, \alpha^{w}_{\pm k} \otimes \alpha^{w}_{\mp k} ) = \exp\Big( \sum\limits_{k=0}^{\infty}\, n_k\, \alpha^{w}_{\pm k} \otimes \alpha^{w}_{\mp k} \Big)
\ee
\subsection{Fock space representations \label{focks}}
Set $\textbf{R}=\Q(t_1^{1/2},t_2^{1/2},a)$. Let $\fock(a)=\textbf{R}[p_1,p_2,\cdots]$ be the set of polynomials in infinitely many variables which we consider as $\textbf{R}$ vector space. Let $\nu$ be a partition and $P_{\nu}\in \fock(a)$ be the corresponding Macdonald polynomials in Haiman's normalization \cite{Haim}\footnote{The parameters $t$, $q$ of Macdonald polynomials are related to $t_1$ and $t_2$ by $q=t_1^{1/2}$, $t=t_2^{1/2}$}. Recall that $P_\nu$ is a basis of the vector space $\fock(a)$.

Let us define a homomorphism $\ev_{a}: \gt \to \textrm{End}\Big(\fock(a)\Big)$ defined explicitly in the basis of Macdonald polynomials by:
\begin{itemize}
\item
On central elements:
\be\label{fi}
\ev_a(K_{(1,0)})\, P_\nu=t_1^{-\frac{1}{2}} t_2^{-\frac{1}{2}} P_\nu, \ \ \ev_a(K_{(0,1)})\,P_\nu=P_\nu
\ee
\item
On Heisenberg subalgebra with slope $0$:
\be
\label{f1}
\ev_a(e_{(m,0)}) P_\nu \mapsto \left\{\begin{array}{rl}
\dfrac{1}{(t_1^{m/2}-t_1^{-m/2})(t_2^{m/2}-t_2^{-m/2})}\,  p_{-m} P_\nu & m<0 \\
\\
 -m \dfrac{\partial P_\nu}{\partial p_m} &   m>0
\end{array}\right.
\ee
\item On Heisenberg subalgebra with slope $\infty$:
\be
\label{f2}
\ev_a (e_{(0,m)})\, P_\nu =  a^{-m} \textrm{sign}(k) \left( \dfrac{1}{ 1-t_1^{m}} \sum\limits_{i=1}^{\infty} \, t_1^{m (\nu_i-1)} t_2^{m(i-1)}   \right) P_\nu
\ee
\end{itemize}
Note that $e_{0,n}$, $e_{n,0}$, $K_{(0,1)}$ and $K_{(1,0)}$ generate $\gt$ thus the last set of formulas
define a homomorphism $\ev_u$.

The representation $\ev_a$ is called \textit{Fock representation} of $\gt$ evaluated at $a$. The parameter $a$ is called the evaluation parameter of Fock representation. Let us note here that by (\ref{f2}) the generators act in $\fock(a)$ as monomials in the evaluation parameter $a$:
\be
\label{alphaa}
\ev_a(\alpha^{w}_{k})\sim a^{-k n(w)}
\ee
\subsection{Tensor product of Fock representations \label{fockrmat}}
Let us denote
\be
\label{rfock}
\fock^{(r)}=\fock(a_1)\otimes \cdots \otimes  \fock(a_r)
\ee
and define representation of quantum toroidal algebra $\ev^{(r)}:\gt\to \textrm{End}\Big( \fock^{(r)} \Big)$ by:
$$
\ev^{(r)}(\alpha)=\ev_{a_1} \otimes \cdots \otimes \ev_{a_r} \Big( \Delta^{(r)} (\alpha)  \Big), \ \ \alpha\in \gt
$$
where $\Delta$ is a coproduct from Section \ref{hopfsec}. We set:
$$
\Rtot^{(r_1),(r_2)}=\ev^{(r_1)}\otimes \ev^{(r_2)} ( \Rtot )
$$
As shown in \cite{OS}, in the tensor product of Fock spaces, the infinite product (\ref{khtot}) converges in the topology of power series in $a_i$ and $\Rtot^{(r_1),(r_2)}$ is a  well defined element of $\textrm{End}\Big( \fock^{(r_1)} \otimes \fock^{(r_2)} \Big)$ whose matrix coefficients are rational functions of evaluation parameters $a_i$.
%%%%%%%%%%%%%%%%%%%%%%%%%%%55
%%%%%%%%%%%%%%%%%%%%%%%%%%%%% FEW WORDS ABOUR R-MATRIX LIMIT
%%%%%%%%%%%%%%%%%%%%%%%%%%%%%
Let $\ev^{(r)}_{a}=\ev_{a_1 a} \otimes \cdots \otimes \ev_{a_r a}$ be a shifted evaluation.
Denote $\Rtot^{(r_1),(r_2)}(a)=\ev^{(r_1)}\otimes \ev^{(r_2)}_a ( \Rtot )$ and $R^{\pm,(r_1),(r_2)}_{w}(a)=\ev^{(r_1)}\otimes \ev^{(r_2)}_a ( R^{\pm }_{w} )$.
\begin{Proposition}(Section 2 of \cite{OS})

\label{Rlimit}
\noindent
\begin{itemize}
\item
The operator $R^{-,(r_1),(r_2)}_{0}\stackrel{def}{=}R^{-,(r_1),(r_2)}_{0}(a)$ does not depend on $a$.
\item
$R^{-,(r_1),(r_2)}_{w}(0)=1$ for $w>0$ and $R^{+,(r_1),(r_2)}_{w}(0)=1$ for $w<0$
\item
$
R^{(r_1),(r_2)}_{\infty}(0)=\hbar^{-\Omega}
$
where $\hbar^{\Omega}$ acts on $\fock^{(r_1)}_{(n_1)}\otimes \fock^{(r_2)}_{(n_2)}$ by multiplication on $\hbar^{(n_1 r_2 +n_2 r_1)/2}$.
\item
In particular, we have:
$
\Rtot^{(r_1),(r_2)}(0) =  (R^{-,(r_1),(r_2)}_{0})^{-1} \hbar^{\Omega}.
$
\end{itemize}
\end{Proposition}
The universal $R$-matrix $\Rtot$ satisfies the quantum Yang-Baxter equation. Thus, its limit $\Rwal=\hbar^{-\Omega} R^{-}_{0}$
is also a solution of this equation. It is not difficult to see that $\Rwal$ is a universal $R$-matrix for $\fh$.

\subsection{Triangular and block-diagonal operators \label{defsec}}
Let us introduce some convenient notations and terminology here. The Fock space is equipped with the grading:
$$
\fock(a)=\bigoplus_{n=0}^{\infty}\, \fock_{(n)}(a)
$$
corresponding to the degree function $\deg(p_k)=k$. This induces the grading
on the tensor products:
$$
\fock^{(r)}=\bigoplus_{n=0}^{\infty}\, \fock^{(r)}_{(n)}
$$
where $\fock^{(r)}_{(n)}$ denotes the subspace spanned by elements of degree $n$. Throughout this paper we use the following terminology: an operator $A : \fock^{(r_1)}\otimes  \fock^{(r_2)} \to \fock^{(r_1)}\otimes  \fock^{(r_2)} $ is \textit{lower-triangular} (upper-triangular)  if $A=\bigoplus\limits_{k=0}^{\infty} A_{(k)}$  (respectively $A=\bigoplus\limits_{k=0}^{\infty} A_{(-k)}$)  where $A_{(k)}:\fock^{(r_1)}_{(n_1)}\otimes  \fock^{(r_2)}_{(n_2)} \to \fock^{(r_1)}_{(n_1+k)}\otimes  \fock^{(r_2)}_{n_1-k}$.
We say that the operator $A$ is \textit{strictly} lower or upper-triangular if in addition $A_{(0)}=1$. We say the operator is block-diagonal if $A_{(k)}=0$ for $k\neq 0$. For example, the operators $R^{-}_w$ and $R^{+}_{w}$ defined  by (\ref{wallR})
are strictly lower and upper-triangular respectively.
  Finally, for a formal variable $z$ we denote a block-diagonal operator $z^{d}_{(i)}$ acting on
a $\fock^{(r_1)}_{(n_1)}\otimes  \fock^{(r_2)}_{(n_2)}$ as multiplication by $z^{n_i}$.

\subsection{Geometric realization of Fock representations}
In the special case $r=1$ the instanton moduli space coincides with the Hilbert scheme of $n$ points on the complex plane $\cM(n,1)=\textrm{Hilb}^{n}(\C^2)$.  As a vector space the equivariant $K$-theory of the Hilbert scheme of points is isomorphic to the Fock space:
\be\label{fc}
\bigoplus\limits_{n=0}^{\infty}\, K_{\bT}(\textrm{Hilb}^{n}(\C^2)  )= \fock(a)
\ee
where $\bT$ is a torus from Section \ref{instantsec}. The summands here correspond to the grading
from the previous section $K_{\bT}(\textrm{Hilb}^{n}(\C^2)  )=\fock_{(n)}(a)$.
The parameters $t_1, t_2$ and $a$ of the Fock space are identified with
the corresponding equivariant parameters of $\bT$.

The fixed point set $\textrm{Hilb}^{n}(\C^2)^{\bT}$ is discrete.
Its elements are labeled by partitions $\nu$ with  $|\nu|=n$. The structure sheaves of the fixed points ${\cal{O}}_{\nu}$ form a basis of the localized $K$-theory. The polynomials representing the elements of this basis
under isomorphism (\ref{fc}) are the Macdonald polynomials $P_{\nu}$ in Haiman normalization \cite{Haim}.

In \cite{SchifVas} the geometric action of $\gt$ on the equivariant $K$-theory of $\textrm{Hilb}^{n}(\C^2)$ was constructed.
As a representation, this action coincides with the Fock representation described in Section \ref{focks}. In particular, formulas (\ref{fi})-(\ref{f2}) describe this geometric action explicitly in the basis of the fixed points.

\subsection{Coproducts and stable envelopes \label{coprs}}
In \cite{Neg,NegFlags} the geometric action of $\gt$ on $K_{\bT}(\cM(r))$ was constructed. As a representation $K_{\bT}(\cM(r))$ is isomorphic to a product of $r$ Fock spaces $\fock^{(r)}$. The evaluation parameters $a_i$ of this representation are identified with
the equivariant parameters of framing torus $\bA$. An alternative construction of $\gt$ action on equivariant $K$-theory
was given in \cite{OS} using  FRT formalism \cite{FRT}. The key ingredient of \cite{OS} is the notion of the stable map,
which is a geometric description of the coproduct. Here we recall basic facts about the stable map, the details can be found in Section 2 of \cite{OS}.

Assume that a one-dimensional subtorus $\bC\subset \bA$ splits the framing as $r$ to $r_1+a r_2$ such that
$\cM(r)^{\bC}=\cM(r_1)\times\cM(r_2)$. In this case we have a canonical isomorphism of the corresponding $K$-theories, called \textit{stable envelope}\footnote{To be precise, the definition of stable envelope map requires a choice of a chamber
in  $\fC \subset Lie(\bC)$ and an alcove $\nabla\subset Pic(\cM(n,r))\otimes \Q$.
In this paper the chamber $\fC$ corresponds to $a\to 0$
and $\nabla$ unique alcove lying in the opposite of the ample cone whose closure contains $0\in H^{2}(\cM(n,r),\R)$. This is the same choice as in Theorem 3 in \cite{OS}. The other choices of alcove $\nabla$ correspond to a freedom in the choice of the Hopf structure, i.e., the antipode $\Delta$ for $\gt$. }:
$$
K_{\bT}(\cM(r)^{\C^*})=K_{\bT}(\cM(r_1))\otimes K_{\bT}(\cM(r_1)) \stackrel{\Stab}{\longrightarrow} K_{\bT}(\cM(r))
$$
The stable map $\Stab$ and the antipode $\Delta$ make the following diagram commutative:
\[
\xymatrix{
 K_{\bT}(\cM(r_1))\otimes K_{\bT}(\cM(r_1)) \ar[d]^{\Delta(\alpha)}\ar[rr]^{\ \ \ \ \ \Stab} & & K_{\bT}(\cM(r))  \ar[d]^{\alpha}\\
 K_{\bT}(\cM(r_1))\otimes K_{\bT}(\cM(r_1)) \ar[rr]^{ \ \ \ \ \Stab} && K_{\bT}(\cM(r))
}
\]
for every element $\alpha \in \gt$.

Let us consider a chain of splittings by whole $\bA$, such that all factors $r_i=1$ in the end:
\be
\label{splt}
r \to a_1 r_1+a_2 r_2 \to a_1 r_1+a_2 r_2 +a_3 r_3 \to\cdots \to  a_1+\cdots+a_r
\ee
This gives a canonical isomorphism of $\gt$-modules:
\be
\label{isof}
K_{\bT}(\cM(1))\otimes\cdots \otimes K_{\bT}(\cM(1)) =\fock^{(r)} \stackrel {\Phi^{(r)}} \longrightarrow K_{\bT}(\cM(r))
\ee
where $\fock^{(r)}$ is defined by (\ref{rfock}), and $\Phi^{(r)}$ is a composition of the stable envelopes corresponding to splittings.
By construction of stable envelope, $\Phi^{(r)}$ does not depend on the order of splittings in (\ref{splt}) and thus it is well defined.
This property reflects the coassociativity of coproduct.

%The solution of the quantum difference equation has the following $a\to 0$ behaviour:
%\begin{Theorem}
%$$
%\Psi^{(r)}(z,0)= J^{(r_1),(r_2)}_{0}(z \times ???) \, \Psi^{(r_1)}(z \hbar^{r_2}) \otimes \Psi^{(r_2)}(z \hbar^{-r_1})
%$$
%with $J^{(r_1),(r_2)}_{0}(z) \in \gt^{\otimes 2}\!\!\!\!(z)$ given explicitly by the following universal formula:
%$$
%J^{(r_1),(r_2)}_{0}(z)=\exp\Big( -\sum\limits_{k=1}^{\infty}\, \dfrac{(\hbar-\hbar^{-1}) K^k_{(1,0)}\otimes  K^{-k}_{(1,0)  }}{1-z^{-k} %K_{(1,0)^k}\otimes  K^{-k}_{(1,0) }} \,  \alpha^{0}_{-k} \otimes \alpha^{0}_{k}  \Big)
%$$
%\end{Theorem}

%%%%%%%%%%%%%%%%%%%%%%%%%%%%%%%%%%%%%%%%%%%%%%%%%%%%%%%%%%%%%%%%%%%%%%%%
%%%%%%%%%%%%%%%%%%%%%%%%%%%%%%%%%%%%%%%%%%%%%%%%%%%%%%%%%%%%%%%%%%%%%%%%%%%%%%%%%%%%%%%%%%%%%%%
%%%%%%%%%%%%%%%%%%%%%%%%%%%%%%%%%%%%%%%%%%%%%%%%%%%%%%%%%%%%%%%%%%%%%%%%%%%%%%%%%%%%%%%%%%%%%%%
%%%%%%%%%%%%%%%%%%%%%%%%%%%%%%%%%%%%%%%%%%%%%%%%%%%%%%%%%%%%%%%%%%%%%%%%%%%%%%%%%%%%%%%%%%%%%%%
%%%%%%%%%%%%%%%%%%%%%%%%%%%%%%%%%%%%%%%%%%%%%%%%%%%%%%%%%%%%%%%%%%%%%%%%%%%%%%%%%%%%%%%%%%%%%%%

\section{Quantum difference equations \label{qdedefsec}}
\subsection{The quantum difference operator \label{qdeopsec}}
First let us consider an element $B(z)\in \gt[[z,q]]$ given explicitly by:
$$
B(z)=\prod\limits_{{w \in \Q} \atop { -1\leq w<0 }}^{\leftarrow}\!\!\!:\!\exp\Big( \sum\limits_{k=0}^{\infty} \dfrac{n_k \,\hbar^{- k r d(w)/2}}{1-z^{-k d(w)}q^{k n(w)} \hbar^{-k r d(w)/2} }\,  \alpha^{w}_{-k} \alpha^{w}_{k} \Big)\!\!:.
$$
The symbol $::$ here denotes the normal ordered exponent - all annihilation operators $\alpha^{w}_{k}$ with $k>0$ are moved to the right. We identify $K_{\bT}(\cM(r))$ with representation $\fock^{(r)}$ of $\gt$ through isomorphism  $\Phi^{(r)}$.
Define the following operator acting on $K_{\bT}(\cM(r))$ :
\be
\gM^{(r)}(z)={\cal{O}}(1)\, \ev_{a_1}\otimes\cdots\otimes \ev_{a_1} \Big( \Delta^{\otimes r} (B(z))   \Big)
\ee
where ${\cal{O}}(1)$ is the operator of multiplication by the corresponding line bundle in $K$-theory. Note that the coefficients
 of normally ordered exponentials act in  $\fock^{(r)}$ as a locally nilpotent operators. This means that the operators
$\gM^{(r)}(z)$ has rational matrix coefficients for all $r$. The qde for the rank $r$ instanton moduli space is given by the following equation \cite{OS}:
\be
\label{qde}
\Psi^{(r)}(z q) {\cal{O}}(1)= \M^{(r)}(z) \Psi^{(r)}(z).
\ee
The capping operator in quantum $K$-theory is a fundamental solution of qde, i.e., corresponds to the following boundary condition\footnote{Here $1 \in K_{\bG}(\cM(n,r))$ stands for the class of the structure sheaf ${\cal{O}}_{\cM(n,r)}$. }:
\be
\label{bc}
\Psi^{(r)}(0)=1
\ee
which geometrically means that the classical limit $z=0$ of the capped descendent vertex coincides with the corresponding
$K$-theory class.

\subsection{Cocycle identity}
Out goal to to compare the capping operator $\Psi^{(r)}(z)$ for $\cM(r)$ with
the (shifted) capping operator for the $\bC$-fixed point set given by
$\Psi^{(r_1)}(z \hbar^{\frac{r_2}{2}})\otimes \Psi^{(r_2)}(z \hbar^{-\frac{r_1}{2}})$. By definition, the first operator is the fundamental solution of (\ref{qde}) while the second solves
$$\begin{array}{l}
\Psi^{(r_1)} (z  \hbar^{\frac{r_2}{2}} q) \otimes \Psi^{(r_2)} (z \hbar^{-\frac{r_1}{2}} q) {\cal{O}}(1) = \\
 \\ \gM^{(r_1)}(z \hbar^{\frac{r_2}{2}}) \otimes \gM^{(r_2)} (z \hbar^{-\frac{r_1}{2}})
\Psi^{(r_1)} (z \hbar^{\frac{r_2}{2}}) \otimes \Psi^{(r_2)} (z \hbar^{-\frac{r_1}{2}})
\end{array}
$$
We conclude that the operator:
\be
\label{defJ}
\gJ^{(r_1),(r_2)}(z)= \Psi^{(r)} (z) \Big(\Psi^{(r_1)}(z \hbar^{\frac{r_2}{2}}) \otimes \Psi^{(r_2)} (z \hbar^{-\frac{r_1}{2}} )\Big)^{-1}
\ee
is the solution of the following difference equation:
\be
\label{spt}
\textbf{J}^{(r_1),(r_2)}(z q)  \gM^{(r_1)}(z \hbar^{\frac{r_2}{2}} ) \otimes \gM^{(r_2)} (z \hbar^{-\frac{r_1}{2}}) = \gM^{(r)}(z) \gJ^{(r_1),(r_2)}(z)
\ee
\begin{Proposition}
The operator $\gJ^{(r_1),(r_2)}(z)$ satisfies the dynamical cocycle identity:
\be
\label{dci}
\gJ^{(r_1+r_2),(r_3)}(z) \gJ^{(r_1),(r_2)}(z \hbar^{\frac{r_3}{2}})=\gJ^{(r_1),(r_2+r_3)}(z) \gJ^{(r_2),(r_3)}(z \hbar^{-\frac{r_1}{2}})
\ee
\end{Proposition}

\begin{proof}
Consider the chain of splittings:
$$\cM{(r_1+r_2+r_3)} \to \cM{(r_1+r_2)} \times \cM{(r_3)} \to \cM{(r_1)}\times\cM{(r_2)}\times \cM{(r_3)}$$
or, alternatively:
$$
\cM{(r_1+r_2+r_3)} \to \cM{(r_1)} \times \cM{(r_2+r_3)} \to \cM{(r_1)}\times\cM{(r_2)}\times \cM{(r_3)}
$$
The result does not depend on the choice of the order and the proposition follows from definition of $\gJ^{(r_1),(r_2)}(z)$.
\end{proof}
By definition, $\Psi^{(r)}(z)$ is a power series in $z$ whose coefficients are rational functions of $a$.
As we discussed in \cite{OS} the limit $\lim\limits_{a\to 0} \Psi^{(r)}(z)$ exists and is lower-triangular. The power series
$\Psi^{(r_1)} (z \hbar^{\frac{r_2}{2}}) \otimes \Psi^{(r_2)} (z \hbar^{-\frac{r_1}{2}})$ is independent of $a$. We conclude:
\begin{Corollary}
\label{clrc}
The operator
\be
\label{Yop}
Y^{(r_1),(r_2)}(z)=\lim\limits_{a\to 0} \gJ^{(r_1),(r_2)}(z)
\ee is a lower-triangular operator satisfying the dynamical cocycle identity (\ref{dci}).
\end{Corollary}
The difference operators $\gM^{(r)}(z)$ are given by evaluation of a universal elements from $\gt[[z,q]]$ in the Fock representation. It means that the solution $\gJ^{(r_1),(r_2)}(z)$ and $Y^{(r_1),(r_2)}(z)$ are also given by evaluation of some universal elements. In particular $Y^{(r_1),(r_2)}(z)=\ev^{(r_1)}\otimes \ev^{(r_2)}( Y(z)) $ for some universal $Y(z)$. In the following sections we find $Y(z)$ explicitly.

Finally, let us note here that by definition (\ref{defJ}) we have $\gJ^{(r),(0)}(z)=\gJ^{(0),(r)}(z)=1$
and thus:
\be
\label{Yz}
Y^{(r),(0)}(z)=Y^{(0),(r)}(z)=1, \ \ \forall r.
\ee

\subsection{Shift operators}
 The global dependence of the capping operator $\Psi^{(r)}(z)$ on the equivariant parameter $a$ is governed by certain $q$-difference equation dual to qde, the so called quantum Knizhnik-Zamolodchikov equation, \cite{OS} (Section 4.2):
\be
\label{qkz}
\Psi^{(r)}(z,a q) \textsf{E}(z,a) = \textsf{S}(z,a) \Psi^{(r)}(z,a)
\ee
Moreover, under identification (\ref{isof}) the operators are explicitly expressible through
the universal $R$-matrices  from Section \ref{fockrmat}:
$$
\textsf{S}(z,a)=z_{(1)}^{d}  \Rtot^{(r_1),(r_2)}(a), \ \ \  \textsf{E}(z,a)=z_{(1)}^{d} \Big(R_{\infty}^{(r_1),(r_2)}(a)\Big).
$$
\begin{Proposition}
The operator $\Psi^{(r)}(z,0)$ is a lower-triangular solution of wall Knizhnik-Zamolodchikov (wKZ) equation:
\be
\label{wqKZ}
z_{(1)}^{d}  (R^{-,(r_1),(r_2)}_{0})^{-1} \hbar^{\Omega} \Psi^{(r)}(z,0) = \Psi^{(r)}(z,0) \hbar^{\Omega} z^{d}_{(1)}
\ee
\end{Proposition}
\begin{proof}
The existence of a limit $\Psi^{(r)}(z,0)$ was shown in \cite{OS}. The proof follows from
Proposition \ref{Rlimit}.
\end{proof}
Note that the operator $\hbar^{\Omega} z^{d}_{(1)}$  commutes with any block diagonal-operator,
such as for example $\Psi^{(r_1)}(z \hbar^{r_2/2}) \otimes \Psi^{(r_2)} (z \hbar^{-r_1/2})$. Thus, we conclude:
\begin{Corollary} \label{Yqkz}
The operator $Y^{(r_1),(r_2)}(z)$ defined by (\ref{Yop}) is a lower-triangular solution of wKZ equation:
\be
z_{(1)}^{d}  (R^{-,(r_1),(r_2)}_{0})^{-1} \hbar^{\Omega} Y^{(r_1),(r_2)}(z) = Y^{(r_1),(r_2)}(z) \hbar^{\Omega} z^{d}_{(1)}
\ee
\end{Corollary}

%In summary, by Corollaries \ref{clrc} and \ref{Yqkz} the operator $Y^{(r_1),(r_2)}(z)$ is lower-triangular
%operator satisfying the wKZ equation and the cocycle identity.

\subsection{Uniqueness of solutions of wKZ equation}
\begin{Lemma} \label{uniq}
For a given block diagonal operator $D^{(r_1),(r_2)}$ there exists unique lower triangular solution $J^{(r_1),(r_2)}$ of wKZ:
\be
\label{wqKZ}
z_{(1)}^{d}  (R^{-,(r_1),(r_2)}_{0})^{-1} \hbar^{\Omega} J^{(r_1),(r_2)} = J^{(r_1),(r_2)} \hbar^{\Omega} z^{d}_{(1)}
\ee
and diagonal part $J^{(r_1),(r_2)}_{(0)}=D^{(r_1),(r_2)}$.
\end{Lemma}
\begin{proof}
Write the last equation in the form:
$$
Ad_{\hbar^{-\Omega} z^{-d}_{(1)}} (J^{(r_1),(r_2)}) = \tilde{R} J^{(r_1),(r_2)}
$$
with $\tilde{R}=\hbar^{-\Omega} (R^{-,(r_1),(r_2)}_{0})^{-1} \hbar^{\Omega}$.   The operator $\tilde{R}$ is strictly lower triangular, thus, taking the $n$-th component of this equation we obtain:
$$
Ad_{\hbar^{\Omega} z^{-d}_{1}} (J^{(r_1),(r_2)}_{(n)}) = J^{(r_1),(r_2)}_{(n)}+\cdots
$$
where dots stand for terms $J^{(r_1),(r_2)}_{(k)}$ with $k<n$. As the operator $1-Ad_{\hbar^{\Omega} z^{-d}_{1}}$ is invertible for general $z$
we can solve this linear equation for $J^{(r_1),(r_2)}_{(n)}$ recursively through the lower terms $J^{(r_1),(r_2)}_{(k)}$ with $k<n$. By induction, all terms are thus expressed through the lowest term $J^{(r_1),(r_2)}_{(0)}=D^{(r_1),(r_2)}$.
\end{proof}
\begin{Corollary}
The solution of wKZ with $J^{(r_1),(r_2)}_{(0)}=D^{(r_1),(r_2)}$ is given by $J^{(r_1),(r_2)}=E^{(r_1),(r_2)} D^{(r_1),(r_2)}$ where $E^{(r_1),(r_2)}$ is unique strictly lower triangular solution.
\end{Corollary}
\begin{proof}
Indeed, by above Proposition there exist unique strictly lower-triangular solution $E^{(r_1),(r_2)}$.
The operator $\hbar^{\Omega} z^{d}_{(1)}$ commutes with any block diagonal operator $D^{(r_1),(r_2)}$ and thus
$J^{(r_1),(r_2)}=E^{(r_1),(r_2)} D^{(r_1),(r_2)}$ is the solution of wKZ with $J^{(r_1),(r_2)}_{(0)}=D^{(r_1),(r_2)}$.
\end{proof}
Finally, we give explicit universal formula for the strictly lower triangular solution $E$.
\begin{Proposition} \label{propY}
There exist the universal element $E(z)\in \gt^{\otimes 2}\!\!\!\![[z]]$ given explicitly by:
\be
\label{jo}
E(z)=\exp\Big(\sum\limits_{k=1}^{\infty}\,\dfrac{n_k }{1-z^{-k} K^{-k}\otimes K^{k}}\,K^{-k} \alpha_{-k} \otimes K^{k} \alpha_{k} \Big)
\ee
such that $E=\ev^{(r_1)}\otimes \ev^{(r_2)} (E(z))$ is the strictly lower-triangular solution of wKZ equation
in $\fock^{(r_1)}\otimes \fock^{(r_2)}$.
\end{Proposition}
\begin{proof}
To show that $E(z)$ is a universal solution of $wKZ$ it is enough to check that
it solves the following equation in $\gt^{\otimes 2} (z)$:
\be
\label{univqKZ}
z_{(1)}^{-d} E(z) z_{(1)}^{d} = (R^{-}_{0})^{-1} \hbar^{\Omega} E(z) \hbar^{-\Omega}
\ee
The operator $R^{-}_{0}$ is given by (\ref{wallR}). Note, that the operators
$z_{(1)}^{d}$ and $\hbar^{\Omega}$ are only defined in the representation $\fock^{(r_1)}\otimes \fock^{(r_2)}$.
However, the operators $Ad_{z_{(1)}^{d}}$ and $Ad_{\hbar^{\Omega}}$ are well defined in $\gt^{\!\! \otimes 2}$:
$$
z_{(1)}^{-d} \alpha_{-k} \otimes \alpha_{k} z_{(1)}^{d}= z^{-k} \alpha_{-k} \otimes \alpha_{k},
$$
and Lemma \ref{Slemma} below gives:
$$
\hbar^{\Omega} \alpha_{-k} \otimes \alpha_{k} \hbar^{-\Omega}= K^{k} \alpha_{-k} \otimes K^{-k} \alpha_{k}
$$
From (\ref{wallR}) we see that (\ref{univqKZ}) is equivalent to the set of equations:
$$
\dfrac{n_k   K^{-k}\otimes K^{k}   z^{-k} }{1-z^{-k} K^{-k}\otimes K^{k}} =- n_k+  \dfrac{n_k   }{1-z^{-k} K^{-k}\otimes K^{k}}
$$
which are identities.
\end{proof}

\begin{Lemma}   \label{Slemma}
For all $r_1,r_2$ the identity:
$$
\hbar^{\Omega} \alpha_{-k} \otimes \alpha_{k} \hbar^{-\Omega} = K^{-k}\alpha_{-k} \otimes K^{k} \alpha_{k}
$$
holds in $\fock^{(r_1)}\otimes \fock^{(r_2)}$.
\end{Lemma}
\begin{proof}
Note that $\alpha_{-k} \otimes \alpha_{k}: \fock^{(r_1)}_{(n_1)}\otimes \fock^{(r_2)}_{(n_2)}\to \fock^{(r_1)}_{(n_1+k)}\otimes \fock^{(r_2)}_{(n_2-k)}$. Thus, from the explicit action of $\hbar^{\Omega}$ on the fock space described in
Proposition \ref{Rlimit} we obtain:
$$
\hbar^{\Omega} \alpha_{-k} \otimes \alpha_{k} \hbar^{-\Omega} = \hbar^{k (r_2-r_1)/2}  \alpha_{-k} \otimes \alpha_{k}
$$
The central element $K$ acts on $\fock^{(r)}$ by multiplication on a scalar $\hbar^{-r/2}$. The lemma is proven.
\end{proof}
\begin{Corollary} \label{corE}
For all $r_1$ and $r_2$ we have:
$$
Y^{(r_1),(r_2)}(z)=\ev^{(r_1)}\otimes \ev^{(r_2)} ( E(z) D(z) )
$$
where $D(z)$ universal operator such that $D^{(r_1),(r_2)}(z)$ is block diagonal in every representation.
\end{Corollary}

\subsection{Universal form of $Y(z)$ \label{psec}}
Our final step is to prove that  $D=1$ which means:
\begin{Proposition} \label{lastpro}
For all $r_1$ and $r_2$ we have:
$$
Y^{(r_1),(r_2)}(z)=\ev^{(r_1)}\otimes \textrm{ev}^{(r_2)} ( E(z) )
$$
\end{Proposition}
\begin{proof}
By Corollary \ref{corE} we have:
$$
Y^{(r_1),(r_2)}(z)=E^{(r_1),(r_2)}(z) D^{(r_1),(r_2)}
$$
for some block diagonal matrix $D^{(r_1),(r_2)}$ and we need to show that $D^{(r_1),(r_2)}=1$.
By  Corollary \ref{clrc} we have:
\be
\label{cocyc0}
Y^{(r_1+r_2),(r_3)}(z) Y^{(r_1),(r_2)}(z \hbar^{\frac{r_3}{2}})=Y^{(r_1),(r_2+r_3)}(z) Y^{(r_2),(r_3)}(z \hbar^{-\frac{r_1}{2}})
\ee
Let us consider left and right wKZ operators (they are coproducts  of the wKZ operators in the first and second components respectively):
\be
\nonumber
A_L(J)=\Rwal_{23} \Rwal_{13} z^{d}_{(3)} J z^{-d}_{(3)} \hbar^{\Omega_{13}+\Omega_{23}}\\ \nonumber
\\ \nonumber
A_R(J)=\Rwal_{12} \Rwal_{13} z^{-d}_{(1)} J z^{d}_{(1)} \hbar^{\Omega_{12}+\Omega_{13}}
\ee
where $\Rwal$ as in Section \ref{fockrmat}.
It is obvious that the left side of (\ref{cocyc0}) satisfies the equation $A_L(J)=J$. Similarly, the right side satisfies
$A_R(J)=J$. Therefore we must have:
$$
A_R(Y^{(r_1+r_2),(r_3)}(z) Y^{(r_1),(r_2)}(z \hbar^{\frac{r_3}{2}}))=Y^{(r_1+r_2),(r_3)}(z) Y^{(r_1),(r_2)}(z \hbar^{\frac{r_3}{2}})
$$
Taking degree zero part of this equality in the third component we obtain:
$$
\begin{array}{l}
\Rwal^{-}_{12} \hbar^{-\Omega_{13}} z^{-d}_{(1)} D^{(r_1+r_2),(r_3)}  Y^{(r_1),(r_2)}(z \hbar^{\frac{r_3}{2}})    z^{d}_{(1)} \hbar^{\Omega_{12}+\Omega_{13}}\\
\\=D^{(r_1+r_2),(r_3)}  Y^{(r_1),(r_2)}(z \hbar^{\frac{r_3}{2}})
\end{array}
$$
The block-diagonal operator $D^{(r_1+r_2),(r_3)}$ commutes with $\hbar^{-\Omega_{13}}$ and $z^{-d}_{(1)}$ so we can rewrite the last equation in the form:
$$
\begin{array}{l}
\Rwal_{12}  D^{(r_1+r_2),(r_3)} \hbar^{-\Omega_{13}} z^{-d}_{(1)} Y^{(r_1),(r_2)}(z \hbar^{\frac{r_3}{2}})    z^{d}_{(1)} \hbar^{\Omega_{12}+\Omega_{13}}\\
\\=D^{(r_1+r_2),(r_3)}  Y^{(r_1),(r_2)}(z \hbar^{\frac{r_3}{2}})
\end{array}
$$
By Lemma \ref{auxlem} this is equivalent to the equality:
$$
\begin{array}{l}
\Rwal_{12}  D^{(r_1+r_2),(r_3)} (z \hbar^{\frac{r_3}{2}})^{-d}_{(1)} Y^{(r_1),(r_2)}(z \hbar^{\frac{r_3}{2}})    (z \hbar^{\frac{r_3}{2}})^{d}_{(1)} \hbar^{-\Omega_{12}}\\
\\=D^{(r_1+r_2),(r_3)}  Y^{(r_1),(r_2)}(z \hbar^{\frac{r_3}{2}})
\end{array}
$$
By definition the operator $Y^{r_1,r_2}(z)$ solves wKZ equation (\ref{wqKZ}) which gives:
\be
\label{eq1}
\Rwal_{12} D^{(r_1+r_2),(r_3)}(\Rwal_{12})^{-1}=D^{(r_1+r_2),(r_3)}
\ee
The same argument for $A_{L}$ gives similar restriction for the second component of $D$:
\be
\label{eq2}
\Rwal_{23} D^{(r_1),(r_2+r_3)}(\Rwal_{23})^{-1}=D^{(r_1),(r_2+r_3)}
\ee
The universal element has the form:
$
D(z)=\sum_{i,j} f_{i,j}(z) b_{i} \otimes b_{j} \in \gt^{\otimes 2}[[z]]
$
in a basis of degree zero elements $b_i$. The first equation (\ref{eq1}) gives $\Delta^{op}(b_i)=\Delta(b_{i})$.
The only elements of $\gt$ with this properties are $b_i=K^{i}$. Similarly, the second (\ref{eq2}) equation gives $b_j=K^{j}$. We conclude that the universal operator has the following form:
$
D(z)=\sum_{i,j} f_{i,j}(z) K^{i} \otimes K^{j}.
$
The diagonal part of (\ref{Yz}) gives $D^{(r_1),(0)}(z)=D^{(0),(r_2)}(z)=1$ for all $r_1$ and $r_2$.
which means that $f_{0,0}(z)=1$ and $f_{i,j}(z)=0$ of $i\neq 0$ or $j\neq 0$. We conclude that $D=1$.
%%%%%%%%%%%%%%%%%%%%%%%%%%%%%%%%%%%%%%%%%%%%%%%%%%%%%%%%%%%%%%%%%%%%%%%%%%%%%%%%%%%%%%%%%%%%%%%%%%

%which means that diagonal matrix $D^{(r_1),(r_2)}$ acts on $\fock^{(r_1),(r_2)}$ by multiplication
%on a scalar $D$.
%Finally note that $D=Y^{(r_1),(r_2)}(0)$, at the same time, by definition of $Y^{(r_1),(r_2)}(z)$ we have
%$$
%Y^{(r_1),(r_2)}(0)= \left(\Psi^{(r)}(0) (\Psi^{(r_1)}(0)\otimes \Psi^{(r_2)}(0))^{-1}\right)_{a=0}=1
%$$
%where the last equality is the normalization of capping operator (\ref{bc}).
\end{proof}
\noindent
To finish the proof we need the following simple lemma.
\begin{Lemma}  \label{auxlem}
The following identity holds in $\fock^{(r_1)}\otimes \fock^{(r_2)}\otimes \fock^{(r_3)}$
\be
\hbar^{-\Omega_{13}}  (\alpha_{-k}\otimes \alpha_k \otimes 1) \hbar^{\Omega_{13}} = (\hbar^{r_3/2})^{-d}_{(1)} (\alpha_{-k}\otimes \alpha_k \otimes 1) (\hbar^{r_3/2})^{d}_{(1)}
\ee
\end{Lemma}
\begin{proof}
From explicit action of Heisenberg algebra we know that the operator $\alpha_{-k}\otimes \alpha_k \otimes 1$ maps:
$$
\fock^{(r_1)}\otimes \fock^{(r_2)}\otimes \fock^{(r_3)}: \fock^{(r_1)}_{(n_1)}\otimes \fock^{(r_2)}_{(n_2)}\otimes \fock^{(r_3)}_{(n_3)}\to \fock^{(r_1)}_{(n_1+k)}\otimes \fock^{(r_2)}_{(n_2-k)}\otimes \fock^{(r_3)}_{(n_3)}
$$
Thus, from the definition of $\hbar^{\Omega}$ given in Proposition \ref{Rlimit} we obtain:
$$
\hbar^{-\Omega_{13}}  (\alpha_{-k}\otimes \alpha_k \otimes 1) \hbar^{\Omega_{13}} = \hbar^{-k r_3/2} (\alpha_{-k}\otimes \alpha_k \otimes 1)
$$
Similarly, by the definition of $z^d_{(i)}$ from Section \ref{defsec} we obtain the same result
$$
(\hbar^{r_3/2})^{-d}_{(1)} (\alpha_{-k}\otimes \alpha_k \otimes 1) (\hbar^{r_3/2})^{d}_{(1)}=\hbar^{-k r_3/2} (\alpha_{-k}\otimes \alpha_k \otimes 1)
$$
\end{proof}
\subsection{Proof of Theorem \ref{facth}}
The theorem \ref{facth} in now proved. Indeed, by definition of $\gJ^{(r_1),(r_2)}$ (\ref{defJ}) and definition of
$Y^{(r_1),(r_2)}(z)$ (\ref{uY}) we have
$$
\lim\limits_{a\to 0}\,\Psi^{(r)}(z)=Y^{(r_1),(r_2)}(z) \Psi^{(r_1)}(z \hbar^{\frac{r_2}{2}}) \otimes \Psi^{(r_2)} (z \hbar^{-\frac{r_1}{2}} )
$$
By Proposition \ref{lastpro} the operator $Y^{(r_1),(r_2)}(z)$ is an evaluation of the universal element (\ref{jo}), which is the claim of Theorem $\ref{facth}$. $\Box$

\section{Bare vertex \label{vers}}
\subsection{Localization to the fixed points}
In this section we consider the descendent bare vertex defined by (\ref{bare}). The push-forward in (\ref{bare}) can be computed
explicitly by the equivariant localization to the locus of the fixed point set of $\bG$-action on the moduli space of nonsingular quasimaps $\qm^{d}_{p_2}(n,r)$. This provides a formula for the bare vertex as a power series in $z$ with explicit coefficients. Here we sketch its derivation and use it to prove the factorization Theorem \ref{facver}.

First, recall that the fixed set $\qm^{d}_{p_2}(n,r)^{\bG}$ consist of finitely many isolated points $\textbf{p}$ which parameterize the following data:
$$
\qm^{d}_{p_2}(n,r)^{\bG}=\{ \textbf{p}=(\lb,\dd): |\lb|=n, \ \ |\dd|=d \}
$$
where $\lb=(\lambda^{(1)},\cdots,\lambda^{(r)})$ is an $r$-tuple of Young diagrams
with total number of $|\lb|=|\lambda^{(1)}|+\cdots+|\lambda^{(r)}|=n$ boxes. The degree data $\dd=\{d_{\Box}\geq 0 :\Box\in \lb\}$
assigns a non-negative integer to each box of $\lb$ such that $|\dd|=\sum\limits_{\Box\in \lb} d_{\Box}=d$. The degree $\dd$ is assumed to be stable, which means that
the data $d_{\Box}$ for $\Box \in \lambda^{(i)}$ for $i=1,\cdots,r$ define $r$ \textit{plane partitions}.
The equivariant localization formula thus gives the following power series for the bare descendent vertex:
\be
\label{locf}
V^{(\tau)}(z)=\sum\limits_{d=0}^{\infty}\sum\limits_{\textbf{p} \in \qm^{d}_{p_2}(n,r)^{\bG}} \, a(\textbf{p}) z^d \in K_{\bG}(\cM(n,r))_{loc}\otimes \Q[[z]]
\ee
where $a(\textbf{p})$ is a contribution of a fixed point $\textbf{p}$.

The fixed set $\cM(n,r)^{\bT}$ consists of isolated points labeled by $r$-tuples of
Young diagrams $\lb$ with $|\lb|=n$. In fact, this is a special case of the fixed points on the moduli space of quasimaps
(the degree zero or constant quasimaps) $\cM(n,r)=\qm^{0}_{p_2}(n,r)$.  The classes of the fixed points $[\lb] \in K_{\bG}(\cM(n,r))_{loc}$ form a basis in the localized $K$-theory. Assume that in this basis the bare vertex has the expansion $V^{(\tau)}(z)=\sum\limits_{|\lb|=n}\, V^{(\tau)}_{\lb}(z) [\lb]$. In (\ref{locf}) only the fixed points $\fp=(\overrightarrow{\nu},\dd)$ with $\overrightarrow{\nu}=\lb$ contribute to the coefficient $V^{(\tau)}_{\lb}(z)$. Therefore we have:
$$
V^{(\tau)}_{\lb}(z)=\sum\limits_{{d_{\Box} \geq 0}\atop {\Box\in \lb}}\,  a_{\lb}((\lb,\dd)) z^{|\dd|}
$$
where the sum runs over the stable degrees $\dd$.
Before we give an explicit formula for $a_{\lb}((\lb,\dd))$ let us introduce the following auxiliary functions: we set $n(\Box)=k$ if $\Box\in \lambda^{(k)}$ and denote by $x(\Box)$ and $y(\Box)$ the standard coordinates of a box in a partition. Define:
$$
\varphi_{\lb}(\Box)=a_{n(\Box)} t_1^{x(\Box)} t_2^{y(\Box)}
$$
and
\be
\label{tautbun}
\tb(\lb,\dd)=\sum\limits_{\Box \in \lb}\, \varphi_{\lb}(\Box)\, q^{d_{\Box}}, \ \ \tw=a_1+\cdots+a_r.
\ee
For each fixed point we define a Laurent polynomial in equavariant parameters $S(\lb,\dd)\in K_{\bG}(pt)$ by:
\be
\label{Sfun}
\begin{array}{r}
S(\lb,\dd)=\tw^{*} \tb(\lb,\dd) + \tw \, \tb(\lb,\dd)^{*} t_1^{-1} t_2^{-1}\\
\\-\tb(\lb,\dd)^{*} \,\tb(\lb,\dd)(1-t_1^{-1})(1-t_2^{-1}).
\end{array}
\ee
The symbol $*$ corresponds to taking dual in $K$-theory and acts by inversion of  all weights $*: w \mapsto w^{-1}$.
\subsection{Virtual tangent space}
The $\bG$-character of the virtual tangent space $T^{vir}_{(\lb,dd)} \qm^{d}_{p_2}(n,r)$  is given explicitly by \cite{pcmilect}:
\be
T^{vir}_{(\lb,\dd)} \qm^{d}_{p_2}(n,r) = S(\lb, \overrightarrow{0})+\dfrac{S(\lb,\dd)-S(\lb, \overrightarrow{0})}{q-1}
\ee
where  $\overrightarrow{0}=\{d_{\Box}=0:\, \forall\, \Box \in \lb \}$.
The stability of $\dd$ implies that $T^{vir}_{(\lb,\dd)} \qm^{d}_{p_2}(n,r)$  is a Laurent polynomial without a constant term and with both positive and negative coefficients (negative coefficients appear because it is a \textit{virtual} tangent space).

As we noted above $\cM(n,r)=\qm^{0}_{p_2}(n,r)$ and thus
the $\bT$-character of the tangent space to the instanton moduli space at a point $\lb\in \cM(n,r)^{\bT}$ is given by
$$
T_{\lb} \cM(n,r) = S(\lb, \overrightarrow{0})
$$
This character is a Laurent polynomial with positive coefficients and no constant term.
\subsection{Bare descendent vertex explicitly}
Let $\tau\in K_{GL(n)}= \Lambda[x_1^{\pm 1},x_2^{\pm 1},\cdots,x_n^{\pm 1}]$ a symmetric polynomial representing the descendent.
We define the evaluation of such polynomial at  fixed point $(\lb,\dd)$ by $\tau(\lb,\dd)=\tau({x_\Box=\varphi_{\lb}(\Box) q^{d(\Box)} })$ (recall that the $r$-tuple of Young diagrams $\lb$ corresponding to a fixed point $(\lb,\dd)$ has exactly $|\lb|=n$ boxes so we may think that the variables $x_i$ are labeled by boxes in $\lb$).

The localization in the equivariant $K$-theory gives the coefficients in (\ref{locf})\footnote{In this formula $q^{-\frac{r|d|}{2}}$
is the contribution of polarization to virtual structure sheaf -  see (6.1.9) in \cite{pcmilect}. In fact, as we discussed in \cite{OS}
the capping operator is given by the fundamental solution of qde in which we should substitute $z\to (-1)^r z$. In this paper, however, we prefer to move this sign to the vertex, in which case all final formulas look simpler. This corresponds to the appearance of factor the ${(-1)^{|d| r}}$ in (\ref{bdv}).   }:
\be
\label{bdv}
V_{\lb}^{(\tau)}(z)=\sum\limits_{{d_{\Box} \geq 0}\atop {\Box\in \lb}}\, \, \hat{a}\Big( T^{vir}_{(\lb,\dd)} \qm^{d}_{p_2}(n,r) \Big) \tau(\lb,\dd)\, (-1)^{|d| r} q^{-\frac{r|d|}{2}} z^{|d|}
\ee
where following Section 3.4.40 of \cite{pcmilect} we use a ``roof'' function defined on Laurent polynomials without a constant term by:
$$
\hat{a}(x+y)= \hat{a}(x) \hat{a}(y), \ \ \ \hat{a}(x)=\dfrac{1}{x^{\frac{1}{2}}-x^{-\frac{1}{2}}}.
$$

\subsection{Proof of the Theorem \ref{facver}}
The coefficients of the power series for the bare descendent vertex (\ref{bdv}) are rational functions of the equivariant parameters
$a_1,\cdots, a_r$ corresponding to framing torus $\bA$. The action of subtorus $\bC\subset \bA$ corresponds to the substitution of framing characters
\be\label{star}
(a_1,\cdots,a_{r_1}, a_{r_1+1} \cdots, a_r)\to (a_1,\cdots,a_{r_1}, a a_{r_1+1} \cdots, a a_r  ).\ee We are interested in the limit $\lim\limits_{a\to 0}\, V^{(r),(\tau)}(z)$ under this substitution. Let $\lb=(\lambda^{(1)},\cdots,\lambda^{(r_1)}
,\lambda^{(r_1+1)}\cdots, \lambda^{(r)})$ be a fixed point on $\cM(n,r)$. Denote by $\lb_1=(\lambda^{(1)},\cdots,\lambda^{(r_1)})$ and $\lb_2=(\lambda^{(r_1+1)}\cdots, \lambda^{(r)})$
the corresponding classes on $\cM(n,r_1)$ and $\cM(n,r_2)$. The theorem \ref{facver} is equivalent to the following
factorization of the coefficients of the bare vertex:
\be
\lim\limits_{a\to 0}\, V^{(r),(\tau)}_{\lb}(z)=V^{(r_1),(\tau)}_{\lb_1}(z \hbar^{\frac{r_2}{2}} ) V^{(r_2),(1)}_{\lb_2}(z
\hbar^{-\frac{r_1}{2}} q^{-r_1})
\ee
To prove it, first note that after substitution (\ref{star}) for the functions (\ref{tautbun}) we obtain:
$$
\tb(\lb,\dd)=\tb(\lb_1,\dd_1)+a \tb(\lb_2,\dd_2), \ \ \ \tw = \tw_1+a \tw_2
$$
where $\tw_1=a_1+\cdots a_{r_1}$ and $\tw_2=a_{r_1+1}+\cdots+ a_{r}$.   Let us consider the contribution to the vertex of the term $\tb(\lb,\dd) \tb(\lb,\dd)^{*}$ in (\ref{Sfun}).
The contribution of a weight $w \in \tb(\lb,\dd) \tb(\lb,\dd)^{*}$  to $\hat{a}(T^{vir}_{(\lb,\dd)}\qm^{d}_{p_2}(n,r))$ has the following form:
$$
\hat{a}\Big(-w (1-t_1^{-1})(1-t_2^{-1})  \Big)= \dfrac{(w^{\frac{1}{2}}-w^{-\frac{1}{2}})(w^{\frac{1}{2}}t_1^{-\frac{1}{2}}t_2^{-\frac{1}{2}}-w^{-\frac{1}{2}}t_1^{\frac{1}{2}}t_2^{\frac{1}{2}})}
{(w^{\frac{1}{2}}t_1^{-\frac{1}{2}}-w^{-\frac{1}{2}}t_1^{\frac{1}{2}})(w^{\frac{1}{2}}t_2^{-\frac{1}{2}}-w^{-\frac{1}{2}}t_2^{\frac{1}{2}})}
$$
If $w\in \tb(\lb_i,\dd_i) \tb(\lb_j,\dd_j)^{*}$ with $i \neq j$ then $w \sim a^{\pm 1}$ and in the limit $a\to 0$ the last expression is equal to $1$.  Thus, in the limit $a\to 0$ the contribution of the term $\tb(\lb,\dd) \tb(\lb,\dd)^{*}$ to $T^{vir}_{\lb,\dd}\qm^{d}_{p_2}(n,r)$
splits to a sum of contributions of $\tb(\lb_1,\dd_1) \tb(\lb_1,\dd_1)^{*}$ and $\tb(\lb_2,\dd_2) \tb(\lb_2,\dd_2)^{*}$.
Thus, the contribution to the vertex, given by roof function $\hat{a}$ factors to a product of corresponding contributions.

Second, we analyse terms $\tb(\lb_i,\dd_i) \tw_{j}$ with $i\neq j$. We assume $i=1$ and $j=2$.
The corresponding contribution to $S({\lb},\dd)$ has the form:
$$
\begin{array}{l}
\tw_2^{*} \dfrac{\tb(\lb_1,\dd_1)-\tb(\lb_1,\overrightarrow{0})}{q-1}+\tw_2 \dfrac{(\tb(\lb_1,\dd_1)^{*}-\tb(\lb_1,\overrightarrow{0})^{*} )}{(q-1)t_1 t_2}=\\
\\
\sum\limits_{j=1}^{r_2} \sum\limits_{\Box\in \lb_1}\,\Big( \dfrac{ \varphi_{\lb}(\Box) }{a a_j} \sum\limits_{i=0}^{d(\Box)-1}\, q^{i}  -\dfrac{a a_j }{\varphi_{\lb}(\Box) t_1 t_2 q  } \sum\limits_{i=0}^{d(\Box)-1}\, q^{-i}  \Big)
\end{array}
$$
Thus the corresponding contribution to the vertex gives:
$$\begin{array}{l}
\prod\limits_{\Box\in \lb_1}\prod\limits_{j=1}^{r_2}\prod\limits_{i=0}^{d(\Box)-1}\,
 \frac{\Big(\frac{a a_j }{\varphi_{\lb_1}(\Box) q^i  t_1 t_2 q }\Big)^{1/2}- \Big(\frac{a a_j }{\varphi_{\lb_1}(\Box) q^i  t_1 t_2 q }\Big)^{-1/2}}{\Big(\frac{\varphi_{\lb_1}(\Box) q^{i}}{a a_j}\Big)^{1/2} -\Big(\frac{\varphi_{\lb_1}(\Box) q^{i}}{a a_j}\Big)^{-1/2} }
\\
\\
\stackrel{a\to 0}{\longrightarrow}\prod\limits_{\Box\in \lb_1}\prod\limits_{j=1}^{r_2}\prod\limits_{i=0}^{d(\Box)-1}\,
\Big(-t_1^{1/2} t_2^{1/2} q^{1/2}\Big)=(-(\hbar q)^{1/2})^{r_2 |\dd_1|}
\end{array}
$$
Same calculation for $i=2, j=1$  gives $(-(\hbar q)^{1/2})^{-r_1 |\dd_2|}$.
The terms $\tb(\lb_1,\dd_1) \tw_{1}$ and $\tb(\lb_2,\dd_2) \tw_{2}$ do not depend on $a$.  We conclude, that the contribution
of $\tb(\lb,\dd) \tw$ to the vertex in the limit $a\to 0$ factors to the contributions of  $\tb(\lb_1,\dd_1) \tw_{1}$ and  $\tb(\lb_2,\dd_2) \tw_{2}$ shifted by some powers of
$-(\hbar q)^{1/2}$. Finally, we have:
$$
\lim\limits_{a\to 0} \varphi_{\lb}(\Box)=\left\{\begin{array}{ll} \varphi_{\lb}(\Box) & \textrm{if} \ \  \Box \in \lb_{1} \\ 0 & \textrm{if} \ \  \Box \in \lb_{2}\end{array}\right.
$$
which for a symmetric polynomial $\tau$ means $\lim\limits_{a\to 0} \tau(\lb,\dd )=\tau(\lb_1,\dd_1)$
Overall, we obtain:
$$
\begin{array}{l}
\lim\limits_{a\to 0} V^{(r),(\tau)}_{\lb}(z)= \\
\\
\sum\limits_{\dd_1, \dd_2}  \hat{a}\Big( T^{vir}_{(\lb_1,\dd_1)} \qm^{d_1}_{p_2}(n_1,r_1) \Big)\hat{a}\Big( T^{vir}_{(\lb_2,\dd_2)}  \qm^{d_2}_{p_2}(n_2,r_2) \Big)\\
\\ \times
(-(\hbar q)^{\frac{1}{2}})^{r_2 |\dd_1|}(-(\hbar q)^{\frac{1}{2}})^{-r_1 |\dd_2|} z^{|\dd_1| + |\dd_2|} \tau(\lb_1,\dd_1) (-1)^{|d| r} q^{-\frac{|d| r}{2}} \\
\\
=  V_{\lb_1}^{(r_1),(\tau)}\Big(z \hbar^{\frac{r_2}{2}}\Big)V^{(r_2),(1)}_{\lb_2}\Big( z \hbar^{-\frac{r_1}{2}} q^{-r_1} \Big)
\end{array}
$$
$\Box$

\bibliographystyle{abbrv}

\bibliography{bib}

\def\cprime{$'$} \def\cprime{$'$}
\begin{thebibliography}{10}

\bibitem{Haim}
M.~Haiman.
\newblock Notes on {M}acdonald polynomials and the geometry of {H}ilbert
  schemes.
\newblock In {\em Symmetric functions 2001: surveys of developments and
  perspectives}, volume~74 of {\em NATO Sci. Ser. II Math. Phys. Chem.}, pages
  1--64. Kluwer Acad. Publ., Dordrecht, 2002.

\bibitem{MOOP}
D.~Maulik, A.~Oblomkov, A.~Okounkov, and R.~Pandharipande.
\newblock Gromov-{W}itten/{D}onaldson-{T}homas correspondence for toric
  3-folds.
\newblock {\em Invent. Math.}, 186(2):435--479, 2011.

\bibitem{NegFlags}
A.~Negu{\c{t}}.
\newblock Moduli of flags of sheaves and their {$K$}-theory.
\newblock {\em Algebr. Geom.}, 2(1):19--43, 2015.

\bibitem{Neg}
A.~Negut.
\newblock {\em Quantum {A}lgebras and {C}yclic {Q}uiver {V}arieties}.
\newblock ProQuest LLC, Ann Arbor, MI, 2015.
\newblock Thesis (Ph.D.)--Columbia University.

\bibitem{pcmilect}
A.~Okounkov.
\newblock {Lectures on K-theoretic computations in enumerative geometry}.
\newblock {\em ArXiv: 1512.07363}.

\bibitem{OPQuanCoh}
A.~Okounkov and R.~Pandharipande.
\newblock Quantum cohomology of the {H}ilbert scheme of points in the plane.
\newblock {\em Invent. Math.}, 179(3):523--557, 2010.

\bibitem{OP}
A.~Okounkov and R.~Pandharipande.
\newblock The quantum differential equation of the {H}ilbert scheme of points
  in the plane.
\newblock {\em Transform. Groups}, 15(4):965--982, 2010.

\bibitem{OS}
A.~Okounkov and A.~Smirnov.
\newblock {Quantum difference equation for Nakajima varieties}.
\newblock {\em ArXiv: 1602.09007}, 2016.

\bibitem{PPToric}
R.~Pandharipande and A.~Pixton.
\newblock Descendent theory for stable pairs on toric 3-folds.
\newblock {\em J. Math. Soc. Japan}, 65(4):1337--1372, 2013.

\bibitem{PPRationality}
R.~Pandharipande and A.~Pixton.
\newblock Descendents on local curves: rationality.
\newblock {\em Compos. Math.}, 149(1):81--124, 2013.

\bibitem{FRT}
N.~Y. Reshetikhin, L.~A. Takhtadzhyan, and L.~D. Faddeev.
\newblock Quantization of {L}ie groups and {L}ie algebras.
\newblock {\em Algebra i Analiz}, 1(1):178--206, 1989.

\bibitem{SchifVas}
O.~Schiffmann and E.~Vasserot.
\newblock The elliptic {H}all algebra and the {$K$}-theory of the {H}ilbert
  scheme of {$\Bbb A^2$}.
\newblock {\em Duke Math. J.}, 162(2):279--366, 2013.

\end{thebibliography}

\newpage

\noindent
Andrey Smirnov\\
Department of Mathematics\\
University of California, Berkeley\\
Berkeley, CA 94720-3840, U.S.A\\
smirnov@berkeley.edu

\end{document}